\documentclass{article}
\usepackage{graphicx} 
\usepackage{authblk}
\usepackage{algorithm}
\usepackage{algpseudocode}
\usepackage[a4paper, total={5.65in, 9in}]{geometry}
\usepackage{placeins}
\usepackage{mathrsfs}
\usepackage{tikz}
\usepackage{hyperref}
\usepackage{amsmath}
\usepackage{float}
\usepackage{amssymb}
\usepackage{amsthm}
\def\C{\mathbb{C}}
\def\Z{\mathbb{Z}}
\def\Q{\mathbb{Q}}
\def\OO{\mathcal{O}}

\title{\textbf{{\Large Enhanced Algorithms for the Representation of integers by Binary Quadratic forms: Reduction to Subset Sum}}}
\date{}
\author{\small MAHER MAMAH}
\affil{{\small Department of Mathematics, The Pennsylvania State University, University Park, USA}}
\newtheorem{theorem}{Theorem}[section]   
\newtheorem{lemma}[theorem]{Lemma}       

\usepackage{titlesec}
\titleformat{\section}
  {\small\fontsize{12}{16}\selectfont\centering} 
  {\thesection.}{1em}{}
\usepackage{hyperref}
\usepackage{arydshln} 
\usepackage[utf8]{inputenc}
\usepackage{abstract}
\usepackage{mathrsfs}
\usepackage{titlesec}
\titleformat{\subsection}[runin]
  {\normalfont\large}  
  {\thesubsection.}{1em}{}
\newcommand{\orthsum}{\mathbin{\text{\footnotesize$\bigcirc$\kern-0.87em$\perp$}}}
\begin{document}
\maketitle
\begin{center}
\begin{minipage}{0.9\textwidth}
    {\small\textbf{Abstract.}  In this paper, we give efficient algorithms for solving the Diophantine equation \( f(x, y) = m \) for arbitrary  definite binary quadratic form \( f\), given the factorization of \( m \). While Cornacchia’s algorithm to solve $x^2+dy^2=m$ provides an efficient method in many cases, its running time becomes exponentially large when \( m \) is highly composite, and inherits some subtleties when generalizing to arbitrary form $f$. To address these issues, we show how to reduce the problem to an instance of the Subset-Sum, a weakly NP-complete problem, allowing for more efficient solutions. Leveraging this approach, we develop deterministic algorithms that adapt to different cases based on \( \mathrm{disc}(f) \) and \( m \). In particular, when \( |\mathrm{disc}(f)| = \mathrm{polylog}(m) \), we provide a polynomial-time solution that remains efficient regardless of the structure of \( m \). For more general cases, we present an algorithm that improves upon Cornacchia’s method, achieving a quadratic speedup. Recently, the problem of representation by a form $f$ found important applications in elliptic curves and isogeny-based cryptography, where these algorithms are central to solving norm form equations.}
\end{minipage}
\end{center}
\section{\textbf{Introduction}}
Problems arising in algebraic number theory have long been an interest of the mathematical community due to their rich theoretical structure and their diverse applications in cryptography, coding theory, and beyond. One interesting problem that received much attention since the time of Gauss is that of finding solutions to Diophantine equations, and particularly quadratic forms. The simplest of all is the Pell-type equation $x^2+dy^2=m$ where $d$ and $m$ are positive integers, and finding integral solutions to this equation was believed to be at least as hard as factoring $m$. 

In 1908, Cornacchia \cite{Cor08} gave a deterministic algorithm with $\mathcal{O}(\log^2m)$ time complexity that solves this problem when $m$ is a prime. His method begins by computing square roots \(r^2 \equiv -d \pmod{m}\) and then performs a \(\gcd\)-like iterative process until the output is less than \(\sqrt{m}\) satisfying certain conditions. His approach was later adapted to handle composite values of \(m\) with only minor modifications, given the prime factorization of $m$. However, this generalization to handle composite \(m\) introduces a significant challenge when \(m\) is highly composite, i.e. has ``too many" prime factors. Since the algorithm computes square roots until one yields the correct solution—or, in the worst case, determines that no solution exists—it may be necessary to iterate over all \(2^{\omega(m)}\) possible square roots. As \(\omega(m)\) (the number of distinct prime factors of \(m\)) can be as large as \(\mathcal{O}(\log m/\log\log m)\), the number of iterations may become exponential in \(\log m\) before the algorithm terminates. This has proven problematic in many applications, particularly in isogeny-based cryptography, where no efficient classical or quantum algorithm currently exists to circumvent this issue.
For instance, one of the important tasks in isogeny-based cryptography is to solve the quaternionic analogue of the \textit{supersingular isogeny path problem} of two supersingular elliptic curves $E$, and $E'$. Namely given a left $\mathrm{End}(E)$-ideal $I$ of norm $N$, where End($E$) denotes the endomorphism ring of $E$ in a quaternion algebra, find another equivalent ideal of some norm $d$. It turns out that this problem reduces to solving $q_I(x_1,x_2,x_3,x_4)=d$ where $q_I$ is a quaternary integral positive definite form. The approach to solve this $q_I(x)=d$ is generally heuristic by choosing fixed values for $x_1$ and $x_2$ and then solve the remaining binary quadratic form using Cornacchi's algorithm. This is the main technique used in [\cite{KLPT14}, \cite{ACD24}, \cite{BKM23}, \cite{EL24}]. \\

\textbf{Contributions:} In this paper, we propose multiple algorithms that outperform Cornacchia's algorithm for all orders of $d$ and $m$. The approach to that is leveraging a new reduction of solving the Diophantine quadratic equation $f(x,y)=m$ to an instance of a group subset sum problem. Algorithm \ref{alg1}, which works for arbitrary $m$, is mainly efficient in the case $\mathrm{disc}(f)=\mathrm{polylog}(m)$, and generally has lower time complexity, compared to the Cornacchia's algorithm, for $|\mathrm{disc}(f)|\ll 2^{2\omega(m)}$. Algorithm \ref{alg2}, and theorem \ref{alg3} work for generic values of $d$ and $m$ and achieve a quadratic speed-up over Cornacchia's algorithm using the meet-in-the-middle approach. \\

The paper is structured as follows. Section \ref{section2} provides background information on algebraic number theory and binary quadratic forms. Section \ref{section3} examines the performance of Cornacchia's algorithm and its generalization. Section \ref{section4} introduces the subset sum problem and presents some of the most effective generic algorithms for solving it. Finally, Section \ref{section5} discusses the reduction of our Diophantine equation to a group instance of the subset sum problem and introduces the enhanced algorithms.

\section{\textbf{Background on Algebraic number theory}}\label{section2}
A number field $K$ can be defined as a subfield of the complex numbers $\C$ that is generated over the field of rational numbers $\Q$ by an algebraic number $\alpha$; i.e $K=\Q(\alpha)$, where $\alpha$ is a root of an irreducible polynomial with rational coefficients. The degree of of $K$ over $\Q$, denoted $[K:\Q]$, is the dimension of $K$ as a $\Q$-vector space, and $\mathrm{disc}(K)$ is called the discriminant of $K$.

For a field $K$, the set of elements that are roots of monic irreducible polynomials with integer coefficients forms a ring, called the ring of integers of $K$, and denoted $\mathcal{O}_K$. The ring of integers is essentially a generalization of $\Z\subset \Q$. For a quadratic field $K=\Q(\sqrt{d})$, where $d$ is square-free, the corresponding ring of integers is \[\mathcal{O}_K=\begin{cases}
    \Z[\sqrt{d}] &\text{if ~} d\equiv 2,3\pmod 4\\
    \Z[\frac{1+\sqrt{d}}{2}]& \text{if } d\equiv1\pmod 4.
\end{cases}\]
The ring of integers \( \mathcal{O}_K \) need not be a unique factorization domain (UFD), however, it remains a Dedekind domain. Consequently, every  $\mathcal{O}_K$-ideal factors uniquely as a product of prime ideals. Specifically, any nonzero ideal \( \mathfrak{a} \subseteq \mathcal{O}_K \) can be written uniquely as

\[
\mathfrak{a} = \mathfrak{p}_1^{e_1} \mathfrak{p}_2^{e_2} \cdots \mathfrak{p}_r^{e_r},
\]

where the \( \mathfrak{p}_i \) are distinct prime ideals in \( \mathcal{O}_K \).\\
For a rational prime $p\in\Z$, the principal ideal $(p)$ factors uniquely as $(p)=\mathfrak{p}_1^{e_1} \mathfrak{p}_2^{e_2} \cdots \mathfrak{p}_r^{e_r}$ where each $e_i$ is called the ramification index of $\mathfrak{p}_i$. We say 
\begin{itemize}
    \item[-] $p$ ramifies if, for at least one $\mathfrak{p}_i$, $e_i>1$.
    \item[-] $p$ remains inert, if $r=1$ and $e_1=1$.
    \item[-] $p$ splits if $r>1$ and $e_i=1$ for all $i$.
\end{itemize}
For a quadratic field \( K = \mathbb{Q}(\sqrt{d}) \), the behavior of primes is determined by the Legendre symbol:
a prime \( p \) ramifies \textit{iff} \( p \mid d \), splits \textit{iff} \( \left(\frac{d}{p}\right) = 1 \), and remains inert \textit{iff} \( \left(\frac{d}{p}\right) = -1 \).\\

Given a quadratic order $\mathcal{O}\subset \mathcal{O}_K$, the index $[\mathcal{O}_K:\mathcal{O}]=f$ is called the conductor and $\mathrm{disc}(\mathcal{O})=f_K^2\mathrm{disc}(K)$. We have unique factorization of $\mathcal{O}$-ideal $\mathfrak{a}$ only if $\mathfrak{a}$ and $f$ are coprime, i.e $(\mathrm{N}(\mathfrak{a}),f)=1$, where we call $\mathfrak{a}$ an invertible proper ideal.\\

To measure the potential failure of element-wise unique factorization, one utilizes the concept of ideal class group \( \mathrm{Cl}(\mathcal{O}) \). Two nonzero fractional ideals \( \mathfrak{a} \) and \( \mathfrak{b} \) are considered equivalent if 
$
\mathfrak{a} = \alpha \cdot \mathfrak{b}
$ for some \( \alpha \in K^\times \). By looking at proper invertible integral $\mathcal{O}$-ideals modulo principal ideals, the set of equivalence classes forms a finite abelian group under ideal multiplication

\[
\mathrm{Cl}(\mathcal{O}) = \{ [\mathfrak{a}] \colon \mathfrak{a} \subset \mathcal{O}, \text{ a fractional ideal} \},
\]
with the identity given by the class of all principal ideals. Using the fundamental theorem of finitely generated abelian groups, we have
\[\mathrm{Cl}(\mathcal{O})\cong \bigoplus_{i=1}^n \Z/m_i\Z\]
where $m_i|m_{i+1}$. The order of \( \mathrm{Cl}(\mathcal{O}) \) is the class number \( h(\mathcal{O}) \) which is, by analtyic methods, at most $O(d^{1/2+\varepsilon})$ for arbitrary $\varepsilon>0$.\\

\textbf{Binary Quadratic forms:} A binary quadratic form is a function of the form 
\(
f(x,y) = ax^2 + bxy + cy^2,
\)
where \( a, b, c \in \mathbb{Z} \) and \( a \) and \( c \) are nonzero. We will also utilize the handy notation $(a,b,c)$ to denoted the aforementioned form. A form is primitive if \( a, b, \) and \( c \) are relatively prime. 

The discriminant of such a form is given by
\(
D = b^2 - 4ac.
\)
A form is called positive definite if the discriminant is negative. Positive definite forms are precisely those such that \( f(x,y) \) takes only positive values when \( (x,y) \neq (0,0) \).

For this paper, all quadratic forms considered will be positive definite, primitive, binary quadratic forms.

Two forms \( f \) and \( g \) are properly equivalent if there exist integers \( p, q, r, s \) such that 
\[
f(x,y) = g(px + qy, rx + sy),
\]
and the determinant condition holds, i.e.
\(
ps - qr = 1.
\)
One can consider \( g \) to be obtained from an action of a matrix
\[
\begin{pmatrix} p & q \\ r & s \end{pmatrix} \in SL_2(\mathbb{Z})
\]
on \( f \), where the discriminant is fixed. Proper equivalence defines an equivalence relation on the set of binary quadratic forms of some fixed discriminant. Every form \( f \) within a given class is equivalent to a unique reduced form \( g \), which yields a canonical representation.

Let \( \mathrm{Cl}(D) \) be the set of primitive, positive definite, binary quadratic forms of discriminant \( D \) considered up to proper equivalence. Gauss defined a group law on \( \mathrm{Cl}(D) \), under which \( \mathrm{Cl}(D) \) is isomorphic to the ideal class group \( \mathrm{Cl}(\mathcal{O}) \), where \( \mathcal{O} \) is a quadratic imaginary order of discriminant \( D \). Specifically, we have the following theorem. 
\begin{theorem}[\cite{Cox11}, Theorem 7.7]
    Let \( \mathcal{O} \) be the order of discriminant \( D \). If 
\(
\mathfrak{a} = \mathbb{Z} \alpha + \mathbb{Z} \beta
\)
is an invertible \( \mathcal{O} \)-ideal, then the map
\[
\iota: \mathfrak{a} \mapsto f(x,y) = \frac{N(x\alpha + y\beta)}{N(\mathfrak{a})}
\]
induces an isomorphism between $\mathrm{Cl}(\mathcal{O})$ and $\mathrm{Cl}(D)$. Additionally, the inverse map is given by \[
ax^2 + bxy + cy^2 ~\mapsto~\mathbb{Z} a + \mathbb{Z} \frac{-b + \sqrt{D}}{2}
\]
\end{theorem}

Since this will come up in our applications, its worthy to note that the group law on the set of quadratic forms is defined by Dirichlet composition and it can be computed efficiently, along with the unique reduced form (see \cite{Cohen}, Algorithms 5.4.2 \& 5.4.7).
\section{\textbf{Cornacchia's algorithm}}\label{section3}
Cornacchia's algorithm \cite{Cor08} is a well-known efficient algorithm for solving the Diophantine equation $x^2+dy^2=m$ where $(d,m)=1$ are square free. Cornacchia originally addressed the problem of solving the equation where $m$ is a prime, but his method actually generalizes to composite $m$ via the Chinese remainder theorem, given the factorization of $m$.

The method proceeds by computing square roots $r_i^2\equiv-d \pmod m$ where WLOG $r_i<m/2$, if no such roots exist, then the equation is not solvable over the integers. To find a solution, we use the Euclidean algorithm by computing \( r_1 \equiv m \pmod{r_0} \), followed by \( r_2 \equiv r_0 \pmod{r_1} \), and continuing this process iteratively until we reach a remainder \( r_k \) such that \( r_k < \sqrt{m} \). At this point, we check whether 
\[
s = \sqrt{\frac{m - r_k^2}{d}}
\]
is an integer. If it is, then the solution is given by \( x = r_k \) and \( y = s \). Otherwise, we attempt another root of \(-d\) until either a solution is found or all possible roots have been exhausted. If no solution is found, then there is no primitive solution.\\

The algorithm is certainly efficient, with a time complexity of \( O(\log^2 m) \), provided that \( m \) is not highly composite. However, when \( m \) has many distinct prime factors, we have \( \omega(m) = O(\log m / \log\log m) \), meaning that iterating over all possible square roots until a solution is found—or worse, determining that no solution exists—causes the algorithm's running time to reach a worst-case complexity of  
\(
O(2^{\omega(m)}\log m)
\).

Another subtlety arises in generalizing the method for representing \( m \) by an arbitrary quadratic form \( f(x,y) = ax^2 + bxy + cy^2 \) with discriminant \( \mathrm{disc}(f) = -D < 0 \). The method proceeds as follows: multiplying both sides of the equation \( ax^2 + bxy + cy^2 = m \) by \( 4a \) and completing the square; we obtain  

\[
(2ax + by)^2 + D y^2 = 4am.
\]

Defining \( X = 2ax + by \) and \( Y = y \), we apply Cornacchia’s algorithm to find integer solutions \( X, Y \) satisfying  

\begin{equation}\label{completesquare}
    X^2 + D Y^2 = 4am,
\end{equation}

where \( X = 2ax + bY \) and \( y = Y \). However, finding a solution to \( X^2 + D Y^2 = 4am \) is not sufficient; one might iterate over all possible solutions to determine a valid solution for the original equation.  

Another challenge arises when \( a \) is highly composite. In such cases, the algorithm again exhibits exponential time complexity, even in the case where $\omega(m)=1$, i.e. $m$ is a prime.
\section{\textbf{The Subset-Sum problem}}\label{section4}
The subset sum problem is a well-known problem in theoretical computer science which is defined, in the general case, as follows
\begin{quote}
    \textit{Given a set $X=\{x_1,x_2,\cdots,,x_n\}$ and a target sum $T$, decide whether there exists a subset of $X$ that sums to $T$. Equivalently, does there exist $\delta_i\in\{0,1\}$ such that $\sum_{i=1}^n \delta_ix_i=T$.} 
\end{quote}
This problem is known to be {NP-complete}, and the best generic algorithm, due to Schroeppel and Shamir \cite{SS81}, solves it in \( O(2^{n/2} n/2) \)-time and \( O(2^{n/4}) \)-space. Other algorithms exist, but they are probabilistic and can only determine whether a solution exists; they are not efficient for proving that no solution exists, see for instance \cite{HJ10} and \cite{BBS20}. However, the problem is, in fact, weakly NP-complete.
This is due to Bellman \cite{Bel57} who introduced a dynamic programming approach that solves the problem in \( \mathcal{O}(nT) \) time and \( \mathcal{O}(T) \) space, where \( n \) is the number of elements and \( T \) is the target sum.
By defining a state \((i, s)\) to represent the existence of a subset of \(X\) with \(i\) elements summing to \(s\), we can transition to the next states \((i+1, s)\) and \((i+1, s + x_{i+1})\), corresponding to either excluding or including the next element \(x_{i+1}\) in the sum. This forms a state-space tree, which can be explored using the Breadth-First Search (BFS) algorithm. If the state \((n, T)\) is found, we can backtrack to reconstruct the subset summing to \(T\); otherwise, we conclude that no such subset exists.\\

Another version of the subset-sum problem, asks to find $\delta_i\in\{-1,+1\}$ such that $\sum_{i=1}^n\delta_ix_i=T$ for some target $T$. The problem can be essentially reduced to the original one by defining $\delta_i=2\varepsilon_i-1$, where $\varepsilon_i \in \{0,1\}$, and proceeding as before.\\
Another variant, which we will be interested in, would be the following: given vectors $v_1,\cdots, v_n\in G=\bigoplus_{i=1}^k\Z/m_i\Z$, find $\delta_i\in \{-1,+1\}$ such that $\sum_{i=1}^n\delta_iv_i=\vec{T}\in G$. This is essentially the same problem but now we're interested in an element (vector) in this finitely generated abelian group rather than integers. We will show how to use this Dynamic programming method to this instance to efficiently solve our Diophantine equation when $d=\mathrm{polylog}(m)$.
\section{\textbf{Reduction to the Subset-Sum problem}}\label{section5}
In this section we give the algorithms to find a representation of $f(s,t)=m$ where $f$ is a given primitive integral positive definite binary quadratic form of discriminant $D=-f_K^2d$ where $d>0$ is square-free, which we will assume for the rest of the paper. Notice that we can assume $m$ to be square-free since for $m=k^2m'$ a representation for square-free $m'=f(x,y)$ implies $m=f(kx,ky)$. We now start with a classical lemma on representations by quadratic forms and ideals of a given norm.
\begin{lemma}\label{lemma4.1}
  Let \( m  \) be a positive square-free integer, and let \( D \) be the discriminant of a quadratic form \( f(x,y) = ax^2 + bxy + cy^2 \). let $\OO$ be an order of discriminant $D$ in an imaginary quadratic field, then $f(x,y)=ax^2+bxy+cy^2$ represents $m$ if and only if $m$ is the norm of some $\mathcal{O}$-ideal $\mathfrak{a}$ in the ideal class in $\mathrm{Cl}(\mathcal{O})$ corresponding to $f$. Moreover, given generators of the $\OO$-ideal $\mathfrak{a}=\Z\alpha+\Z\beta$, one can efficiently find $x,y\in\Z$ such that $f(x,y)=m$.
\end{lemma}
\begin{proof}
    The proof of the statement and the efficiency of extracting a representation follows from the proof of theorem 7.7 in \cite{Cox11}, as the representation just requires simple linear algebra in dimension 2.
\end{proof}
Next, we give another criterion for the representation of $m$ by $f(x,y)$ which will be the key to our reduction. Note for local conditions and unique factorization of ideals in a quadratic order $\mathcal{O}$ of discriminant $D=-f_K^2d$, we will require $(f_K,m)=1$
\begin{theorem}\label{theorem4.2}
  Let \( m = p_1 \cdots p_r \) be a positive square-free integer, and let \( D = -f_K^2 d \) be the discriminant of the quadratic form \( f(x,y) = ax^2 + bxy + cy^2 \), where \( d>0 \) is square-free and \( (f_K, m) = 1 \). Let $\OO$ be an order of discriminant $D$ in an imaginary quadratic field $K$, then $f(x,y)=ax^2+bxy+cy^2$ represents $m$ if and only if each prime $p_i$ either splits or ramifies in $K$ and \[\prod_{i}[\mathfrak{p}_i]^{\delta_i}\prod_{j}[\mathfrak{p}_j]=[\mathfrak{a]}\] for some $\delta_i\in \{-1,+1\}$ where $[\mathfrak{a}]$ is the ideal class in $\mathrm{Cl}(\OO)$ corresponding to $[f]$ in $\mathrm{Cl}(D)$; each $\mathfrak{p}_k$ is a prime $\mathcal{O}$-ideal lying above $p_k$ with $\mathfrak{p}_i$ and $\mathfrak{p}_j$ split and ramified respectively.
\end{theorem}
\begin{proof}
    Suppose that the equation
\(
f(x,y)=ax^2+bxy+cy^2= m
\)
is solvable over the integers. Then, by following the same procedure as for \ref{completesquare}, the fact that $(f_K,m)=1$ imply that
\(
s^2 \equiv -d \pmod{m}
\)
is solvable, and $-d$ must be a quadratic residue modulo each prime \(p_k\) dividing \(m\), including trivially $0$ when $p_k|d$. Consequently, each \(p_k\) either splits or ramifies in the quadratic field \(K\) of discriminant $-d$.

Let $(p_i)_{i=1}^r$ be the set of primes that split in $K$, and $(p_j)_{j=r+1}^s$ the primes that ramify. Combining this with our previous lemma, there exist an $\OO$-ideal $\mathfrak{a}$ in the corresponding class of $[f]$ of norm $m$, and we obtain the factorization:
\begin{equation} \label{eq1}
    \mathfrak{a}\bar{\mathfrak{a}} = (m) = \prod_{i=1}^{r} \mathfrak{p}_i \bar{\mathfrak{p}}_i \prod_{j=r+1}^s\mathfrak{p}_j^2,
\end{equation}
where \(\mathfrak{p}_k = \mathfrak{P}_k \cap \mathcal{O}\), and \(\mathfrak{P}_k\) is a prime \(\mathcal{O}_K\)-ideal with \(\mathrm{N}(\mathfrak{P}_k)=\mathrm{N}(\mathfrak{p}_k)=p_k\) (see Proposition 7.20 of \cite{Cox11}).
 Notice that $[\bar{\mathfrak{p}_i}]=[\mathfrak{p}_i]^{-1}$ since $p_i\mathcal{O}=\mathfrak{p}_i\bar{\mathfrak{p}_i}$. Additionally, since the norm satisfies \( \mathrm{N}(\mathfrak{a}) = m \) and \( \mathrm{N}(\mathfrak{p}_k) = p_k \), and given that the ring \( \mathcal{O} \) enjoys unique factorization of ideals (due to \( (f_K, m) = 1 \)), it follows from equation (\ref{eq1}) that
    \(\mathfrak{a}=\prod_{i=1}^r\mathfrak{q}_i \prod_{j=r+1}^s \mathfrak{p}_j\) where $\mathfrak{q}_i=\mathfrak{p}_i~\mathrm{or}~\bar{\mathfrak{p}_i}$. This, in return implies
    \[\prod_{i}[\mathfrak{p}_i]^{\delta_i}\prod_{j}[\mathfrak{p}_j]=[\mathfrak{a]}\]
    where $\delta_i\in \{-1,+1\}$.

    On the other hand, if each $p_i$ either splits or ramifies in $K$ with the same notation for each $(p_i)_{i=1}^r$ and $(p_j)_{j=r+1}^s$, and $\prod[\mathfrak{p}_i]^{\delta_i}\prod_{j}[\mathfrak{p}_j]=[\mathfrak{a]}$ in $\mathrm{Cl}(\mathcal{O})$, then $\mathfrak{a}=\prod_{i=1}^r \mathfrak{q}_i\prod_{j}\mathfrak{p}_j$ has norm $m$ which lies in the corresponding ideal class of $[f]$, where $\mathfrak{q}_i=\mathfrak{p}_i~\mathrm{or}~\bar{\mathfrak{p}_i}$. Hence we can find a representation efficiently using lemma \ref{lemma4.1}
\end{proof}
We see that the $\delta_i's \in \{-1,+1\}$ are precisely the coefficients as in the subset sum problem. With this equivalence in hand, we now give an algorithm to find a solution to $f(x,y)=m$ or decide that non exist, which works best for $D^{1/2+\varepsilon}\ll_\varepsilon 2^{\omega(m)}$.
\begin{theorem}
    Given positive square-free integers $m$ and $d$ and binary quadratic form $f$ of discriminant $D=-f_K^2d$, along with the factorization of $m$ such that $(f_K,m)=1$, algorithm \ref{alg1} either finds integral solution to $f(x,y)=m$ or decides that no such solution exists. The algorithm has deterministic time complexity of
    \[O\left(D^{1/2+\varepsilon}\omega(m)+\log^2m\right)\] 
    and $O\left(D^{1/2+\varepsilon}\right)$-space complexity for arbitrary $\varepsilon>0$, where $\omega(m)$ is the number of distinct prime factors of $m$.
\end{theorem}
\begin{proof}
    \textit{Steps 1-5.} These steps check if each prime $p_i\mid m$ splits in $\Q(\sqrt{-d})$, equivalently that the Legendre symbol is $+1$, which takes a total time $\mathcal{O}(\log^2m)$ (ramified primes are easy to check by division).\\
    
    \textit{Step 8.} This step involves computing the class group of the order $\mathcal{O}$ using Shanks algorithm \cite{Sha69}. The algorithm runs unconditionally in time $\mathcal{O}(d^{1/4+\varepsilon})$ for arbitrary $\varepsilon>0$. Even though there exist subexponential algorithms in $\log d$ under GRH, like \cite{HM89} an \cite{Jac99}, for our purposes this will suffice.\\

    \textit{Step 9.} Given the factorization of $m$ we can compute $r_i^2\equiv -d\pmod {p_i}$ for each $i$ using Tonelli-Shanks algorithm in time $\mathcal{O}(\log^2(p_i))$, and hence square roots computation takes an overall $\mathcal{O}(\log^2m)$ time.\\

    \textit{Step 13.} This involves writing each class $[\mathfrak{p}_i]$ in terms of the generators of $\mathrm{Cl}$, this can be done by solving discrete logs in the group. Each class representation can be done in subexponential time.\\
    
     \textit{Steps 23-24.} A state $(i,\Vec{s})$ is a valid state if there exist exists $\delta_1,\cdots,\delta_i$ such that \[\sum_{j=1}^i \delta_i v_i=\Vec{s}\in \prod_{k=1}^r \Z/m_k\Z.\]. Each state $(i,\Vec{s})$ has two next states:
    \begin{itemize}
        \item $(i+1,\Vec{s}+v_{i+1})$ where we choose $\delta_{i+1}=1$ and form the sum $s+v_i$
        \item $(i+1\Vec{s}-v_{i+1})$ where $\delta_{i+1}=-1$ and we form the sum $s-v_{i+1}$,
    \end{itemize}
    and we initialize the state $(0,\Vec{0})$. Since \(\prod_{i}[\mathfrak{p}_i]^{\delta_i}\prod_{j}[\mathfrak{p}_j]=[\mathfrak{a]}\), we see that in vector notation we need $\sum_{i=1}^r \delta_iv_i=a-\sum_{j=r+1}^sv_j=T$.
\FloatBarrier 
\begin{algorithm}[H]
    \caption{N{\footnotesize ORM}R{\footnotesize EPRESENT}($m$,$d$)}\label{alg1}
    \begin{algorithmic}[1]
        \Statex \textbf{Input:} Positive square-free integers $m$ and $d$ along with the prime factorization $m=p_1\cdots  p_s$ and binary quadratic form $f$ of discriminant $D=-f_K^2d$.
        \Statex \textbf{Output:} A vector $(x,y)\in \Z^2$ such that $f(x,y)=m$ or $\perp$.
        \For{$i=1$ to $s$}
        \If{$\left(\frac{-d}{p_i}\right)= -1$}
        \State \Return $\perp$;
        \EndIf
        \EndFor
        \State $K\leftarrow$ Imaginary quadratic field of discriminant $d$;
        \State $\theta, \mathcal{O} \leftarrow$ generator of $\mathcal{O}_K$ and $\mathcal{O}=\Z[f_K\theta]$;
        \State $G=\bigoplus_{k=1}^t \Z/m_k \Z \leftarrow$ group structure $\mathrm{Cl}(\mathcal{O})\cong \bigoplus \Z/m_k \Z $
        \cite{Sha69}
        \For{$p_i$ split in $K$ ($i=1,\cdots r$)}
        \State $r_i\leftarrow$  $r_i^2\equiv -d \pmod {p_i}$ using Tonelli-Shanks algorithm;
        \State  $\mathfrak{P}_i\leftarrow (p_i, r_i+\theta)$;
        \State $\mathfrak{p}_i\leftarrow\mathfrak{P}_i\cap \mathcal{O}$
        \State $v_i \leftarrow$ the representation of $[\mathfrak{p}_i]$ in $G$;
        \EndFor
        \For {$p_j$ ramified in $K$ $(j=r+1,\cdots s$)}
        \State $\mathfrak{P}_j \leftarrow$ prime ideal lying above $p_j$
        \State $\mathfrak{p}_j\leftarrow \mathfrak{P}_j\cap \OO$
        \State $v_j \leftarrow$ the representation of $[\mathfrak{p}_j]$ in $G$;
        \EndFor
        \State $\vec{a} \leftarrow$ representation of ideal class $[a] \in G$ corresponding to $[f]$ 
        \State $T\leftarrow$ $\vec{a}-\sum_{j=r+1}^sv_j$
        \State $V\leftarrow \{v_1,\cdots,v_r\}$;
        \State $(i,\Vec{s})\leftarrow\{\exists~\delta_1,\cdots \delta_i \in \{-1,+1\} \mid \sum_{k=1}^i\delta_k v_k=\Vec{s}\}$;
        \State $(0,\Vec{0})\leftarrow$ initialization state;
        \State Search for the state $(r,T)$ in the state-space tree; \cite{Bel57}
        \If {$(r, T)$ is found}
        \State Backtrack to find $\delta_i$ for $1 \leq i \leq r$
        \Else
           \State \Return $\perp$
        \EndIf
        \State $\mathfrak{a} \leftarrow \prod_{i=1}^{r} \mathfrak{q}_i \prod_{j=r+1}^s\mathfrak{p}_j$ where \(
            \mathfrak{q}_i = 
            \begin{cases}
            \mathfrak{p}_i, & \text{if } \delta_i = +1;\\[1mm]
            \overline{\mathfrak{p}}_i, & \text{if } \delta_i = -1;
            \end{cases}
            \)
        \State $(x,y)\leftarrow$ a representation of $m$ by $f$; $\{$Lemma \ref{lemma4.1}$\}$
        \State \Return $(x,y)$
    \end{algorithmic}
\end{algorithm}
    \textit{Steps 25-30.} Given this state-space tree, we can use the Breadth first Search algorithm (BFS), for instance, to search the state $(r,T)$ as in Bellman's method \cite{Bel57}. If the state is found we can back track and find the sequence of $(\delta_i)_{i=1}^r$ such that $\sum\delta_iv_i=T$. Otherwise, return that no such sequence exists, which implies that \(\prod_{i}[\mathfrak{p}_i]^{\delta_i}\prod_{j}[\mathfrak{p}_j]=[\mathfrak{a]}\) for all combinations of $\delta_i's$ and hence $f(x,y)=m$ is not solvable over $\Z$ by theorem \ref{theorem4.2}. As discussed in section \ref{section4}, this dynamic programming method using (BFS) has time complexity which is at most linear in the number of states. The number of states is $r=\omega(m)$ times the number of sums $\vec{s}$ we obtain, which is bounded above by $|\mathrm{Cl}(\OO)|=h(D)=\mathcal{O}(D^{1/2+\varepsilon})$. Hence the total time complexity for these steps is $\mathcal{O}(d^{1/2+\varepsilon}\omega(m))$. Since space-optimized dynamic programming does not require storing all previous rows, but only maintains an array of distinct sums while keeping track of the current and previous state arrays, the space complexity is reduced to \(\mathcal{O}(h(D)) = \mathcal{O}(D^{1/2+\varepsilon})\).\\

    \textit{Steps 31,32.} After figuring the right combination of $(\delta_i)_{i=1}^{r}$ we have \(\prod_{i}\mathfrak{p}_i^{\delta_i}\prod_{j}\mathfrak{p}_j=\mathfrak{a}\) where $\mathfrak{a}$ lies in the corresponding class of $[f]$ and this step just involves writing the ideal in terms of two generators which can be done iteratively. To extract a representation $f(x,y)=m$, we use lemma \ref{lemma4.1}.
\end{proof}
\textit{Remark. } Since the algorithm's running time is dominated by the complexity of finding the correct state in the tree which is $O(d^{1/2+\varepsilon}\omega(m))$, algorithm \ref{alg1} remains more efficient than Cornacchia's algorithm as long as $D^{1/2+\varepsilon}\ll2^{\omega(m)}$ for arbitrary $\varepsilon>0$.\\

We now describe an algorithm that exploits theorem \ref{theorem4.2} and the meet-in-the-middle approach to achieve a quadratic speedup over Cornacchia's algorithm in the generic case.
\begin{theorem}
   Given positive square-free integers $m$ and $d$ and binary quadaratic form $f$ of discriminant $D=-f_K^2d$, along with the factorization of $m$ such that $(f_K,m)=1$, algorithm \ref{alg2} either finds integral solution to $f(x,y)=m$ or decides that no such solution exists. The algorithm has a deterministic worst-case time \[O\left(2^{\omega(m)/2}\left({\omega(m)}\log D+\log m\right)+\log^2m\right),\] and $O(2^{\omega(m)/2})$-space complexity.
\end{theorem}
\begin{proof}
    Most of the steps are similar to the previous algorithm, we will just comment on the procedure and the running time of some steps.\\

    \textit{Steps 23-28} After forming the set $S$ of reduced binary form classes $\iota([\mathfrak{p}_i])$, we split the set into two equally large sets $S_1$ and $S_2$ of cardinalates $r_j\approx r/2$. In each set $S_j$ we form all products \[\prod_{i=1}^{r_j}\iota([\mathfrak{p}_i])^{\delta_i}\] corresponding each combination of $\delta_i\in\{-1,+1\}$, reducing each one at a time obtain a unique representative $(a,b,c)$. The product of forms, i.e. composition, and the reduction can be done efficiently in $O(\log D)$, and the total time complexity is $O\left(2^{\omega(m)/2}\cdot\omega(m)\log D\right)$. To sort a set of $n$ elements by lexicographic order, this can be in $O(n\log n)$ time, i.e. for our sets $S_i$ in $O(2^{\omega(m)/2}\omega(m)/2)$ time. \\

    \textit{Steps 30-35} For each form class $[F]=[(a,b,c)]$, we check if $[f]\cdot[G]^{-1}\cdot [F]^{-1}=[H][F]^{-1}=[(a',b',c')]\in L_1$, which can be done in $O(\omega(m)/2)$ time in a sorted list. If we exhaust all classes in $L_2$ and $[H][F]^{-1}\not\in L_1$, then there exists no solution to the Diophantine equation by theorem \ref{theorem4.2}. Otherwise we have 
    \[[f]=[G](\left[H][F]^{-1}\right)[F]=\prod_{j=r+1}^s\iota([\mathfrak{p}_j])\prod_{i=1}^{r_1}\iota([\mathfrak{p}_i])^{\delta_{1,i}}\prod_{i=1}^{r_2}\iota([\mathfrak{p}_i])^{\delta_{2,i}}=\iota\left(\prod_{j=r+1}^s\mathfrak{p}_j\prod_{i=1}^r[\mathfrak{p}_i]^{\delta_i}\right),\]
    where $\delta_1=\delta_{[H][F]^{-1}}$ and $\delta_2=\delta_{[F]}$. Hence, $\delta=(\delta_1,\cdots,\delta_r)$ is the desired combination.\\

    The rest of the steps follow exactly as in algorithm \ref{alg1}.
\end{proof}
\begin{algorithm}[H]
    \caption{G{\footnotesize ENERIC}N{\footnotesize ORM}R{\footnotesize REPRESENT}($m$, $d$)}\label{alg2}
    \begin{algorithmic}[1]
        \Statex \textbf{Input:} Positive square-free integers $m$ and $d$ along with the prime factorization $m=p_1\cdots  p_s$ and binary quadratic form $f$ of discriminant $D=-f_K^2d$.
        \Statex \textbf{Output:} A vector $(x,y)\in \Z^2$ such that $f(x,y)=m$ or $\perp$.
        \For{$i=1$ \textbf{to} $s$}
            \If{$\left(\frac{-d}{p_i}\right) =- 1$}
                \State \Return $\perp$
            \EndIf
        \EndFor
        \State $K\leftarrow$ Imaginary quadratic field of discriminant $d$;
        \State $\theta, \mathcal{O} \leftarrow$ generator of $\mathcal{O}_K$ and $\mathcal{O}=\Z[f_K\theta]$;
        \State $\iota \leftarrow$ the isomorphism $\iota:\mathrm{Cl}(\mathcal{O}) \rightarrow \mathrm{Cl}(-d)$
        \For{$p_i$ split in $K$ ($i=1,\cdots r$)}
        \State $r_i\leftarrow$  $r_i^2\equiv -d \pmod {p_i}$ using Tonelli-Shanks algorithm;
        \State  $\mathfrak{P}_i\leftarrow (p_i, r_i+\theta)$;
        \State $\mathfrak{p}_i\leftarrow\mathfrak{P}_i\cap \mathcal{O}$
            \State Reduce the form class $\iota([\mathfrak{p}_i])$ via Gaussian reduction;
        \EndFor
        \For {$p_j$ ramified in $K$ $(j=r+1,\cdots s$)}
        \State $\mathfrak{P}_j \leftarrow$ prime ideal lying above $p_j$
        \State $\mathfrak{p}_j\leftarrow \mathfrak{P}_j\cap \OO$
         \State Reduce the form class $\iota([\mathfrak{p}_j])$ via Gaussian reduction;
        \EndFor
        \State $[G]\leftarrow$ Reduced product $\prod_{j=r+1}^s\iota([\mathfrak{p}_j])$
        \State $[H] \leftarrow$ Reduced Dirichlet composition of classes $[f]\cdot[G]^{-1}$;
        \State $S \leftarrow \left\{ \iota([\mathfrak{p}_1]), \ldots, \iota([\mathfrak{p}_r])\right \}$
        \State Split $S$ into $S_1$ and $S_2$, each of cardinality $\lfloor r/2 \rfloor$;
        \For{$i=1,2$}
            \State Generate all $2^{r/2}$ class products $[F]$ of $S_i$, each with associated sign vector $\delta_{[F]}$;
            \State $(a,b,c) \leftarrow$ the unique reduced form in the class of $[F]$;
            \State Sort the forms in increasing lexicographic order (by $a$, then $b$, then $c$);
            \State Store the pair $\bigl( [(a,b,c)],\, \delta_{[F]} \bigr)$ in the list $L_i$;
        \EndFor
        \For{$[F] \in L_2$}
            \If{$[H]\cdot[F]^{-1} \in L_1$}
                \State Extract $\delta_{[H]\cdot[F]^{-1}}$ and $\delta_{[F]}$ and go to step \ref{step29};
            \EndIf
        \EndFor
        \State \Return $\perp$
        \State $\delta \leftarrow$ the concatenation $\delta_{[H][F]^{-1}} \,||\, \delta_{[F]}$\label{step29};
        \State $I \leftarrow \prod_{i=1}^{r} \mathfrak{q}_i \prod_{j=r+1}^s\mathfrak{p}_j$, where
            \(
            \mathfrak{q}_i = 
            \begin{cases}
            \mathfrak{p}_i, & \text{if } \delta_i = +1;\\[1mm]
            \overline{\mathfrak{p}}_i, & \text{if } \delta_i = -1;
            \end{cases}
            \)
        \State $(x,y)\leftarrow$ a representation of $m$ by $f$; $\{$Lemma \ref{lemma4.1}$\}$;
        \State \Return $(x,y)$
    \end{algorithmic}
\end{algorithm}
We now present a similar algorithm with improved space complexity that is generally more practical.
\begin{theorem}\label{alg3}
   Given positive square-free integers $m$ and $d$ and binary quadaratic form $f$ of discriminant $D=-f_K^2d$, along with the factorization of $m$ such that $(f_K,m)=1$, algorithm \ref{alg2} can be modified to find integral solution to $f(x,y)=m$ or decide that non exist in deterministic worst-case time \[O\left(2^{\omega(m)/2}\omega(m)\log D+\log^2m\right),\] and $O(2^{\omega(m)/4})$-space complexity.
\end{theorem}
\begin{proof}
    The algorithm is exactly like the previous one, however we use Schroeppel and Shamir's approach \cite{SS81} to reduce the space complexity from $O(2^{\omega(m)/2})$ to $O(2^{\omega(m)/4})$. Additionally instead of sorting the forms in each set of products $L_i$ we can use a hash table that takes, on average, $O(1)$-time to insert and look up, which is often more practical.
\end{proof}
\section{\textbf{Acknowledgment}}
This paper originated as a by-product of the author's research in isogeny-based cryptography, where the representation of an integer by a definite quadratic form with a small discriminant was required but could not be achieved using Cornacchia’s algorithm. To address this gap in the literature, the author developed the present work. The author sincerely thanks Lara Tarkh for valuable discussions and insights.

\end{document}